\documentclass[11pt, a4paper, centertags]{amsart}

\usepackage[all]{xy}

\usepackage{amsfonts}
\usepackage{amssymb, amsmath,amsthm, mathtools}
\usepackage{tabularx, hyperref}
\usepackage{enumerate}

\usepackage{tikz-cd}
\usepackage{microtype}

\makeatletter
\newtheorem*{rep@theorem}{\rep@title}
\newcommand{\newreptheorem}[2]{%
\newenvironment{rep#1}[1]{%
 \def\rep@title{#2 \ref{##1}}%
 \begin{rep@theorem}}%
 {\end{rep@theorem}}}
\makeatother

\newtheorem*{thm*}{Theorem}

\theoremstyle{definition}
\newtheorem*{intro_defi*}{Definition}
\newtheorem*{intro_rem*}{Remark}

\theoremstyle{theorem}
\newtheorem{lemma}{Lemma}
\newtheorem{thm}[lemma]{Theorem}
\newreptheorem{thm}{Theorem}  
\newtheorem{prop}[lemma]{Proposition}
\newreptheorem{prop}{Proposition}  
\newtheorem{cor}[lemma]{Corollary}
\newreptheorem{cor}{Corollary}  

\theoremstyle{definition}

\newreptheorem{defi}{Definition}  

\newreptheorem{quest}{Question}  

\newtheorem{rem}[lemma]{Remark}

\theoremstyle{definition}

\makeatletter
\newcommand\norm{\bBigg@{0.8}}

\makeatother
 
 \newcommand{\indnorm}[2][flex]{\csname #1l\endcsname\|#2%
                                 \csname #1r\endcsname\|\mathclose{}}
                                  \newcommand{\indnorml}[4][flex]{\csname #1l\endcsname\|#2%
                                 \csname #1r\endcsname\|_{#3}^{#4}\mathclose{}}

\newcommand{\R} {\ensuremath {\mathbb{R}}}

\newcommand{\Z} {\ensuremath {\mathbb{Z}}}

\renewcommand{\rho}{\varrho}
\def\phi{\varphi}

\usepackage{color}
\usepackage{pdfcolmk}

\long\def\forget#1{}

\def\widetilde{\tilde}

\makeatletter
\@namedef{subjclassname@2020}{%
  \textup{2020} Mathematics Subject Classification}
\makeatother

\begin{document}

\title[Some remarks on acyclicity in bounded cohomology]{Some remarks on acyclicity in bounded cohomology}
\thanks{Title of the forthcoming published version: \emph{Addendum to ``Amenability and acyclicity in bounded cohomology"}}

\author{Marco Moraschini}
\address{\newline M. Moraschini \newline
Dipartimento di Matematica, Universit\`{a} di Bologna, Bologna, Italia}
\email{marco.moraschini2@unibo.it}

\author{George Raptis}
\address{\newline G. Raptis \newline
Department of Mathematics, Aristotle University of Thessaloniki, 541 24 Thessaloniki, Greece}
\email{raptisg@math.auth.gr}


\begin{abstract}
We show that a surjective homomorphism $\varphi \colon \Gamma \to K$ of (discrete) groups induces an isomorphism
$H^\bullet_b(K; V) \to H^\bullet_b(\Gamma; \varphi^{-1} V)$ in bounded cohomology for all dual normed $K$-modules $V$ if and only if the kernel of $\varphi$ is boundedly acyclic. This complements a previous result by the authors that characterized this class of group homomorphisms as bounded cohomology equivalences with respect to $\R$-generated Banach $K$-modules \cite[Theorem~4.1]{moraschiniraptis}. We deduce a characterization of the class of maps between path-connected spaces that induce isomorphisms in bounded cohomology with respect to coefficients in all dual normed modules, complementing the corresponding result shown previously in terms of $\R$-generated Banach modules \cite[Theorem C]{moraschiniraptis}.  The main new input is the proof of the fact that every boundedly acyclic group $\Gamma$ has trivial bounded cohomology with respect to all dual normed \emph{trivial} $\Gamma$-modules.
  \end{abstract}

\maketitle


This note concerns improvements and complementary results concerning the following characterization of boundedly $n$-acyclic maps shown in our previous work \cite{moraschiniraptis}: 

\begin{thm}[{see \cite[Theorem C]{moraschiniraptis}}]\label{thmC}
Let $f \colon X \to Y$ be a map between based path-connected spaces\footnote{Following the conventions of \cite{moraschiniraptis}, we restrict to topological spaces that admit a universal covering.}, let $F$ denote its homotopy fiber, and let $n \geq 0$ be an integer or $n = \infty$.  We denote by $f_* \colon \pi_1(X) \to \pi_1(Y)$ the induced homomorphism between the fundamental groups. Then the following are equivalent:
\begin{enumerate}
\item[(1)] $f$ is boundedly $n$-acyclic, that is, the induced restriction map
$$H^i_b(f; V) \colon H_b^i(Y; V) \to H_b^i(X; f_*^{-1}V)$$ 
is an isomorphism for $i \leq n$ and injective for $i=n+1$ for every $\R$-generated Banach $\pi_1(Y)$-module $V$. 
\item[(2)] The induced restriction map 
$$H^i_b(f; V) \colon H_b^i(Y; V) \to H_b^i(X; f_*^{-1}V)$$ 
is surjective for every $\R$-generated Banach $\pi_1(Y)$-module $V$ and $0 \leq i \leq n$.
\item[(3)] $F$ is path-connected and $H^i_b(X; f_*^{-1}V) = 0$ for every relatively injective $\R$-generated Banach $\pi_1(Y)$-module $V$ and $1 \leq i \leq n$.
\item[(4)]  $F$ is boundedly $n$-acyclic, that is, $H^0_b(F; \mathbb{R}) \cong \mathbb{R}$ and $H^i_b(F; \mathbb{R}) = 0$ for $1 \leq i \leq n$.
\end{enumerate}
\end{thm}

For a set $S$ and a Banach space~$V$, we write $\ell^{\infty}(S, V)$ for the Banach space of bounded functions $S \to V$ (in the case of $V = \R$ we simply write $\ell^\infty(S)$). For a group $\Gamma$, the \emph{bounded complex} of $\Gamma$ with respect to trivial coefficients in~$V$ (i.e., $V$ regarded as a $\Gamma$-module equipped with the trivial action) is the cochain complex
$$0 \to V \xrightarrow{\delta^0} \ell^{\infty}(\Gamma, V) \xrightarrow{\delta^1} \ell^{\infty}(\Gamma^2, V) \to \cdots \to \ell^{\infty}(\Gamma^n, V) \xrightarrow{\delta^n} \cdots $$
with the following coboundary operator~\cite[Section~1.7]{Frigerio:book}: $\delta^0$ is the zero map, and for all $n \geq 1$ we have
\begin{equation}\label{eq:coboundary}
\begin{aligned}
\delta^n(f)(\gamma_1, \cdots, \gamma_{n+1}) &\coloneqq f(\gamma_2, \cdots, \gamma_{n+1}) \\
&+ \sum_{i = 1}^n (-1)^i f(\gamma_1, \cdots, \gamma_i \gamma_{i+1}, \cdots, \gamma_{n+1}) \\
&+ (-1)^{n+1} f(\gamma_1, \cdots, \gamma_n),
\end{aligned}
\end{equation}
where $f \in \ell^\infty(\Gamma^n, V)$ and $\gamma_1, \cdots, \gamma_n, \gamma_{n+1} \in \Gamma$.
We recall that $\Gamma$ is \emph{boundedly $n$-acyclic} if the bounded cochain complex~$(\ell^\infty(\Gamma^\bullet; \R)\; \delta^\bullet)$ has trivial cohomology in degrees $1 \leq i \leq n$. 


\smallskip

An $\R$-generated Banach $\Gamma$-module is a Banach $\Gamma$-module of the form $\ell^{\infty}(S)$ for some $\Gamma$-set $S$ (the $\Gamma$-action on $\ell^{\infty}(S)$ is given by $\gamma \cdot f (s) = f(\gamma^{-1} s)$ for all $\gamma \in \Gamma$ and $s \in S$). As we remarked in our previous work~\cite[Remark~2.41]{moraschiniraptis}, the notion of an $\R$-generated Banach $\Gamma$-module was motivated by the fact that every boundedly $n$-acyclic group $\Gamma$ satisfies $H^i_b(\Gamma; V) = 0$ for $1 \leq i \leq n$ and every $\R$-generated \emph{trivial} Banach $\Gamma$-module $V$~\cite[Proposition~2.39]{moraschiniraptis}.
It turns out that boundedly $n$-acyclic groups have trivial bounded cohomology with respect to a larger class of coefficient modules (this also answers our question in~\cite[Remark~2.41]{moraschiniraptis}):

\begin{prop} \label{prop:trivial_mod} 
Let $\Gamma$ be a boundedly $n$-acyclic group, $n \geq 1$, and let $V$ be a dual normed trivial $\Gamma$-module. Then $H^i_b(\Gamma; V) = 0$ for every $1 \leq i \leq n$.
\end{prop}

The following lemma is the key observation for the proof of Proposition~\ref{prop:trivial_mod}. We owe the inspiration for this lemma to the recent work of Glebsky, Lubotzky, Monod and Rangarajan~\cite[Section~4.2]{asymptoticbc}. 

We recall that a Banach space $X$ is \emph{injective} if for every Banach space $W$ with a subspace $V \subseteq W$ and a bounded linear map $f \colon V \to X$, there is an extension of $f$ to a bounded linear map $\widetilde{f} \colon W \to X$~\cite{janson2012some} -- equivalently, using the uniform continuity of $f$ and completeness of $X$, we may restrict to the case where $V$ is a closed subspace, that is, $V$ is again a Banach space. The Banach space $\ell^{\infty}(S)$ is injective for any set $S$ (using the Hahn-Banach theorem~\cite[Proposition~2.5.2]{albiac2006topics}). Moreover, by using an embedding of $X$ (i.e., an isomorphism onto its image) into some $\ell^{\infty}(S)$, it follows that $X$ is an injective Banach space if and only if every embedding of $X$ into a Banach space~$W$ has a complement in $W$~\cite[Section~3, pp. 92--93]{goodner1950projections} (see also~\cite[Theorem~2.7 and Corollary~2.8]{janson2012some}).

\begin{lemma} \label{lem:complemented} 
Let $\Gamma$ be a boundedly $n$-acyclic group, $n \geq 1$. For every $1 \leq i \leq n$,  $\textup{im}(\delta^{i-1})$ is a closed complemented subspace of $\ell^{\infty}(\Gamma^{i})$. As a consequence, the short exact sequence
$$0 \to \ker(\delta^i) \to \ell^{\infty}(\Gamma^i) \to \textup{im}(\delta^{i}) \to 0$$
splits (as Banach spaces) for every $1 \leq i \leq n-1$. 
\end{lemma}
\begin{proof}
Since $\Gamma$ is boundedly $n$-acyclic, we have $\textup{im}(\delta^{i-1}) = \ker(\delta^{i})$ for every $1 \leq i \leq n$. Thus, $\textup{im}(\delta^{i-1})$ is a closed subspace of $\ell^{\infty}(\Gamma^i)$. In particular, $\textup{im}(\delta^{i-1})$ is a Banach space for every $1 \leq i \leq n$ ($\textup{im}(\delta^0) = \{0\}$). Then, using the bounded $n$-acyclicity of $\Gamma$ and the open mapping theorem, the canonical bounded linear map
$$\ell^\infty(\Gamma^{i-1}) \slash \textup{im}(\delta^{i-2}) \to \textup{im}(\delta^{i-1})$$ 
is an isomorphism for $2 \leq i \leq n$. 
Note that the quotient of an inclusion $V \subseteq W$ of injective Banach spaces is again injective, since it is isomorphic to the complement of $V$ in $W$ and so the quotient map $W \to W \slash V$ admits a section. Hence we conclude inductively that  $\textup{im}(\delta^{i-1})$ is injective for $1 \leq i \leq n$, as required. It follows that $\textup{im}(\delta^{i-1}) = \ker(\delta^{i})$ is a complemented subspace of $\ell^{\infty}(\Gamma^i)$ and so the inclusion $\ker(\delta^i) \subseteq \ell^{\infty}(\Gamma^i)$ admits a splitting. 
\end{proof}

\noindent \emph{Proof of Proposition \ref{prop:trivial_mod}.} Let $V$ be a dual normed space with the trivial $\Gamma$-action. We need to show that the bounded complex of $\Gamma$ with trivial coefficients in $V$:
\begin{equation*}\label{eq:res:inv:standard}
0 \to V \xrightarrow{\delta^0_V}\ell^\infty(\Gamma, V) \xrightarrow{\delta^1_V} \ell^\infty(\Gamma^2, V) \xrightarrow{\delta^2_V} \ell^\infty(\Gamma^3, V) \xrightarrow{\delta^3_V} \cdots 
\end{equation*}
has vanishing cohomology in degrees $1 \leq i \leq n$.
By assumption, this holds for $V = \R$. For the general case, suppose that $V$ is the topological dual $\mathcal{B}(W, \R)$ of the normed space  $W$ with the trivial $\Gamma$-action. Then we have a natural identification of $\Gamma$-modules
$$
\ell^\infty(\Gamma^\bullet, V) \cong \mathcal{B}(W, \ell^\infty(\Gamma^\bullet))
$$
induced by
$$
f \mapsto \Big( w \mapsto \big((\gamma_1, \cdots, \gamma_\bullet) \mapsto f(\gamma_1, \cdots, \gamma_\bullet)(w) \big)  \Big).
$$
Moreover, these identifications yield isomorphic cochain complexes
\[
(V, \ell^\infty(\Gamma^\bullet, V); \delta^\bullet_V) \cong \Bigg(\mathcal{B}(W, \R) , \mathcal{B}\Big(W, \ell^\infty(\Gamma^\bullet)\Big); \mathcal{B}(W, \delta^\bullet)\Bigg).
\]
By Lemma \ref{lem:complemented}, the inclusion $\textup{im}(\delta^{n-1})=\ker(\delta^n) \subseteq \ell^{\infty}(\Gamma^n)$ has a complement $Q$ and we denote by $q \colon \ell^{\infty}(\Gamma^n) \to Q$ the canonical projection. Note that there is a bijective bounded linear map $Q \to \textup{im}(\delta^n)$. Further, by Lemma \ref{lem:complemented}, the ``modified truncation'' of the cochain complex $(\R, \ell^\infty(\Gamma^\bullet); \delta^{\bullet})$
$$
0 \to \R \xrightarrow{\delta^0=0} \ell^\infty(\Gamma) \to \cdots \to \ell^\infty(\Gamma^{n-1}) \xrightarrow{\delta^{n-1}} \ell^{\infty}(\Gamma^n) \xrightarrow{q} Q \to 0 \to \cdots
$$
is a contractible cochain complex of Banach spaces (it decomposes into a direct sum of split short exact sequences) and applying the functor $\mathcal{B}(W, -)$ (that preserves split short exact sequences) yields again a contractible cochain complex. The cohomology of the latter cochain complex agrees with the cohomology of the original cochain complex $(V, \ell^{\infty}(\Gamma^{\bullet}); \delta^{\bullet}_V)$ in degrees $\leq n$ and the result follows.  
\qed

\medskip

Proposition \ref{prop:trivial_mod} essentially makes it possible to replace $\R$-generated Banach $\Gamma$-modules in \cite{moraschiniraptis} with dual normed 
$\Gamma$-modules (as we had actually done originally in an old preprint version of our paper \cite{moraschiniraptisarxiv}!). At the same time, the restriction to $\R$-generated Banach $\Gamma$-modules and Theorem \ref{thmC} could still be useful as it provides a smaller test set of coefficient modules for the detection of boundedly acyclic maps. On the other hand, Remark~\ref{rem:other:coeff} will show that a similar result cannot hold if we only restrict to $\R$ as coefficients.

The proofs of the following results are the same as~\cite[Theorem~4.1 and Theorem~C]{moraschiniraptis} and together they provide a more complete characterization of boundedly $n$-acyclic maps that also includes dual normed modules.

\begin{thm}\label{thm:main:groups}
Let $\phi \colon \Gamma \to K$ be a homomorphism of discrete groups and let $H$ denote its kernel. Let $n \geq 0$ be an integer or $n=\infty$. Then the following are equivalent:
\begin{enumerate}
\item The induced restriction map
$$
H^i_b(\phi; V) \colon H_b^i(K; V) \to H_b^i(\Gamma; \phi^{-1}V)
$$
is an isomorphism for $i \leq n$ and injective for $i = n+1$ for every dual normed $K$-module $V$. 
\item The induced restriction map
$$H^i_b(\phi; V) \colon H_b^i(K; V) \to H_b^i(\Gamma; \phi^{-1}V)$$
is surjective for $0 \leq i \leq n$ and every dual normed $K$-module $V$.
\item $\phi$ is surjective and $H^i_b(\Gamma; \phi^{-1}V) = 0$ for $1 \leq i \leq n$ and every relatively injective dual normed $K$-module $V$.
\item $\phi$ is surjective and $H$ is a boundedly $n$-acyclic group.
\end{enumerate}
\end{thm}

\begin{thm}\label{main:thm} 
Let $f \colon X \to Y$ be a map between based path-connected spaces,  let $F$ denote its homotopy fiber, and let $n \geq 0$ be an integer or $n=\infty$. We denote by $f_* \colon \pi_1(X) \to \pi_1(Y)$ the induced homomorphism between the fundamental groups. Then the following are equivalent:
\begin{enumerate}
\item The induced restriction map
$$
H^i_b(f; V) \colon H_b^i(Y; V) \to H_b^i(X; f_*^{-1}V)
$$
is an isomorphism for $i \leq n$ and injective for $i = n + 1$ for every dual normed $\pi_1(Y)$-module $V$. 
\item The induced restriction map 
$$H^i_b(f; V) \colon H_b^i(Y; V) \to H_b^i(X; f_*^{-1}V)$$ 
is surjective for $0 \leq i \leq n$ and every dual normed $\pi_1(Y)$-module $V$.
\item$F$ is path-connected and $H^i_b(X; f_*^{-1}V) = 0$ for $1 \leq i \leq n$ and every relatively injective dual normed $\pi_1(Y)$-module $V$.
\item $F$ is boundedly $n$-acyclic, that is, $H^0_b(F; \mathbb{R}) \cong \mathbb{R}$ and $H^i_b(F; \mathbb{R}) = 0$ for $1 \leq i \leq n$.
\end{enumerate}
\end{thm}
\begin{proof}[Proofs of Theorem~\ref{thm:main:groups} and Theorem~\ref{main:thm}]
Concerning Theorem \ref{thm:main:groups}, the implications (1)$\Rightarrow$(2)$\Rightarrow$(3)$\Rightarrow$(4) are the same as for the corresponding implications of \cite[Theorem 4.1]{moraschiniraptis}. The last implication (4)$\Rightarrow$(1) is also shown similarly by applying Proposition \ref{prop:trivial_mod} instead of \cite[Proposition~2.39]{moraschiniraptis}. Theorem \ref{main:thm} follows from Theorem \ref{thm:main:groups} using the Mapping Theorem similarly to the proof of \cite[Theorem C]{moraschiniraptis} (see also the proof of \cite[Proposition 2.11]{moraschiniraptis}). 
\end{proof}

Note that the last item in the previous theorems is independent of the choice of coefficient modules. Thus, combining these with Theorem \ref{thmC}, we conclude that bounded acyclicity of maps (as defined in \cite{moraschiniraptis}) can also be defined in terms of dual normed modules.

\begin{cor}
Let $f \colon X \to Y$ be a map between based path-connected spaces and let $n \geq 0$ be an integer or $n = \infty$. We denote by $f_* \colon \pi_1(X) \to \pi_1(Y)$ the induced homomorphism between the fundamental groups. Then the following are equivalent:
\begin{enumerate}
\item[(a)] The induced restriction map $H^i_b(f; V) \colon H_b^i(Y; V) \to H_b^i(X; f_*^{-1}V)$ is an isomorphism for $i \leq n$ and injective for $i = n + 1$ for every dual normed $\pi_1(Y)$-module $V$. 
\item[(b)] The induced restriction map $H^i_b(f; V) \colon H_b^i(Y; V) \to H_b^i(X; f_*^{-1}V)$
is an isomorphism for $i \leq n$ and injective for $i = n + 1$ for every $\R$-generated Banach $\pi_1(Y)$-module $V$.
\end{enumerate}
\end{cor} 

\begin{rem}\label{rem:other:coeff}
We emphasize that the previous result is specific to the respective classes of coefficient modules. Indeed, if we consider just $\R$ as coefficients or even (semi-)separable coefficients in the sense of Monod~\cite[Section~3.C]{monod10}, then the resulting notion of bounded acyclicity is strictly weaker.

More precisely, Monod proved that the restricted wreath product $F_2 \wr \Z \coloneqq (\bigoplus_{\Z} F_2) \rtimes \Z$ is boundedly acyclic with respect to all dual (semi-)separable coefficients~\cite{monod:thompson}. Also, $\Z$ has trivial bounded cohomology for all dual normed coefficient modules. But the group $\bigoplus_{\Z} F_2$ is not boundedly acyclic (this already fails in degree $2$), so the canonical homomorphism $F_2 \wr \Z \to \Z$ does not satisfy the conditions of Theorem \ref{thm:main:groups}. In fact, one can show (by using cohomological induction) that if we consider the $\R$-generated Banach $(F_2 \wr \Z)$-module $\ell^\infty(\Z)$, then $H_b^2(F_2 \wr \Z; \ell^\infty(\Z)) \neq 0$~\cite[Section~4.3]{monod:thompson}.

In this specific example we also see the different nature of the two classes of coefficient modules. Indeed, given a group~$\Gamma$ the basic example of a semi-separable coefficient module is the space of essentially bounded functions $L^\infty(\Omega)$, where $\Omega$ is a standard Borel probability space with a measurable measure-preserving $\Gamma$-action~\cite[Section~3.C]{monod10}. However, $\R$-generated modules are defined through bounded functions over \emph{arbitrary discrete $\Gamma$-sets}.
\end{rem}

The fact that surjective group homomorphisms induce an injective map in degree $2$ in bounded cohomology~\cite[Example~4.2]{moraschiniraptis} extends to the class of dual normed coefficient modules. 

We also have the following version of \cite[Corollary~4.3]{moraschiniraptis}:

\begin{cor}\label{cor:intro:normal:subgrps:vbcr}
Let $n \geq 1$ be an integer or $n = \infty$. Let $\Gamma$ be a boundedly $n$-acyclic group, let $H \unlhd \Gamma$ be a normal subgroup, and let $\phi \colon \Gamma \to \Gamma \slash H$ be the quotient homomorphism. Then $H$ is boundedly $n$-acyclic if and only if 
$$H^{i}_b(\Gamma; \phi^{-1}V) = 0$$ for every relatively injective dual normed $\Gamma \slash H$-module $V$ and $1 \leq i \leq n$.
\end{cor}

\bigskip

\noindent \textbf{Acknowledgements.} We are grateful to Francesco Fournier-Facio for bringing~\cite{asymptoticbc} to our attention. 
We also thank the referees for their careful reading and useful suggestions.

MM was partially supported by the ERC ``Definable Algebraic Topology" DAT - Grant Agreement n. 101077154. GR was partially supported by SFB 1085 - \emph{Higher Invariants} (University of Regensburg) funded by the DFG.

This work has been funded by the European Union - NextGenerationEU under the National Recovery and Resilience Plan (PNRR) - Mission 4 Education and research - Component 2 From research to business - Investment 1.1 Notice Prin 2022 -  DD N. 104 del 2/2/2022, from title ``Geometry and topology of manifolds", proposal code 2022NMPLT8 - CUP J53D23003820001.
 
\bibliographystyle{amsalpha}
\bibliography{svbib}

\end{document}